\documentclass[11pt, preprint, reqno]{amsart}
\usepackage{graphicx} 
\usepackage{tikz}
\usepackage{pgfplots}
\pgfplotsset{compat=1.18}
\usepackage{mathabx}

\usepackage{enumitem}
\usepackage{amsfonts}
\usepackage{amsmath}
\usepackage{amsthm}
\usepackage{amssymb}
\usepackage{mathrsfs}
 \usepackage{hyperref}
\hypersetup{colorlinks={true},linkcolor={black},citecolor=black}
\usepackage[capitalize]{cleveref}
\usepackage{pst-node}
\usepackage{auto-pst-pdf}
\usepackage{tikz-cd} 

\usepackage{geometry}

\renewcommand{\epsilon}{\varepsilon}

\newcommand{\domain}{\Omega}

\newtheorem{theorem}{Theorem}[section]
\newtheorem{lemma}[theorem]{Lemma}
\newtheorem{corollary}[theorem]{Corollary}
\newtheorem{proposition}[theorem]{Proposition}

\theoremstyle{definition}
\newtheorem{remark}[theorem]{Remark}

\title{Monopolist nonlinear pricing in a seller's market}
\author{Lucas D. O'Brien}
\address{Department of Mathematics, Massachusetts Institute of Technology}
\email{obrie720@mit.edu}
\date{}

\thanks{
\\ 
Lucas O'Brien's work was supported by an internal fellowship from the Massachusetts Institute of Technology School of Science. The author is grateful to Robert McCann, Cale Rankin, and Kelvin Shuangjian Zhang for many fruitful exchanges.  \\
    \copyright \today}

\begin{document}

\begin{abstract} 
    In this paper, we study the Rochet-Chon\'{e} model of monopolist pricing in the seller's market limit. We introduce a monotonicity formula which yields an asymptotic description of solutions, and prove that the proportion of consumers priced out of the market vanishes as overall consumer demand increases, providing a counterpoint to Armstrong's desirability of exclusion. 
\end{abstract}

\maketitle

\section{Introduction}

Given a distribution of consumers, the monopolist's problem is to find a profit-maximizing price menu for a set of available products. This problem serves as a paradigm for multidimensional screening in economics, see \cite[\S 7]{Basov05}. A model case of the monopolist's problem is given by the following variational problem, derived in the seminal work of Rochet and Chon\'{e} \cite{RochetChone98}. Given $n \geq 2$, a bounded, open, and convex domain $\domain \subseteq \mathbf{R}^n$ and a Lipschitz $f \in C^{0,1}(\overline{\domain})$ with $f \geq \lambda > 0$, minimize the functional 
\begin{equation}\label{eq:profitfunctiondefinition}
    \Phi[u] := \int_{\domain}[\frac{1}{2}|D u - x|^2 + u]\, f(x)dx
\end{equation}
over the set
\begin{equation}\label{eq:admissiblefunctions}
    \mathcal{U} := \{u : \overline{\domain} \to \mathbf{R} \ |\ u \geq 0 \text{ is convex}\}. 
\end{equation}
Here, $f\,d\mathcal{H}^n|_{\domain}$ represents the distribution of consumers, 
\[
\Pi[u] := \frac{1}{2}\int_{\domain}|x|^2 f(x)dx - \Phi[u]
\]
the monopolist's profits, and $u(x)$ the \textit{indirect utility} of a consumer $x \in \domain$; mathematically, $u$ is the Legendre transform of the price menu. Its gradient $Du$ maps a consumer $x \in \domain$ to the product they purchase under the monopolist's optimal price menu. By work of Caffarelli and Lions \cite{CaffarelliLions06+} and its extension by McCann, Rankin, and Zhang \cite{McCannRankinZhang25}, the minimizer $u$ of \eqref{eq:profitfunctiondefinition} over \eqref{eq:admissiblefunctions} has the optimal interior regularity 
\[u \in C^{1,1}_{\mathrm{loc}}(\domain).\] 
See the recent paper by Chen, Figalli, and Zhang \cite{ChenFigalliZhang26+} for results on the global regularity of $u$, including a construction showing that $C^1(\overline{\domain})$ regularity is sharp on convex but not strictly convex planar domains when $f \equiv 1$. 

The above variational problem has proven to be incredibly mathematically rich and subtle, with deep connections to the obstacle problem, Monge-Amp\'{e}re equation, and the theory of optimal transport. By and large, this subtlety is explained by the presence of \textit{bunching} and \textit{exclusion} regions in the monopolist's solution. Armstrong's \textit{desirability of exclusion} is the assertion that a positive proportion of consumers should be priced out of the market under the monopolist's optimal strategy; mathematically, this means that the \textit{exclusion region} $Du^{-1}\{0\}$ should have positive volume. Armstrong proved this result when $\domain$ is strictly convex \cite[Proposition 1]{Armstrong96}, while Figalli, Kim, and McCann \cite[Theorem 4.8]{FigalliKimMcCann11} verified the desirability of exclusion when $\domain$ has no $(n-1)$-dimensional facets in its boundary; see also \cite[Theorem 1.1]{McCannRankinZhang24+} for a refinement of the desirability of exclusion when boundary facets are present. The exclusion region is a distinguished example of the phenomenon of \textit{bunching}, in which all consumers in a set $\ell \subseteq \domain$ of positive dimension purchase the same product under the monopolist's optimal price menu. These nontrivial level sets of $Du$ are referred to as \textit{leaves}, or \textit{rays} if they are one-dimensional, see \cite{McCannRankinZhang24+}.

Recently, the groundbreaking work of McCann and Zhang \cite{McCannZhang23+}, \cite{McCannZhang24+} and McCann, Rankin, and Zhang \cite{McCannRankinZhang24+} provided a complete qualitative description of the bunching regions of the family of solutions $u$ to the monopolist's problem on the square 
\[
\domain = [a, a+1]^2 \subseteq \mathbf{R}^2, \quad a \geq 0,
\]
with constant density $f \equiv 1$. This work uncovers a sharp distinction between bunching behaviours in a buyer's market $a\ll 1$ versus in a seller's market $a \gg 1$ \cite[Theorem 1.5]{McCannRankinZhang24+}, and gives an explicit description the exclusion region $Du^{-1}\{0\}$ in a seller's market \cite[Lemma 4.7]{McCannZhang23+}, \cite[Theorems 1.1, 1.5]{McCannRankinZhang24+}. Additionally, this work included the discovery of new bunching regions foliated by rays whose free boundaries are modelled locally by low-regularity obstacle problems; their regularity was studied in subsequent works by the author, McCann, and Rankin \cite{McCannOBrienRankin26} and Chen, Figalli, and Zhang \cite{ChenFigalliZhang26+}. 

Motivated by this work on the monopolist's problem on the square $[a, a+1]^2 \subseteq \mathbf{R}^2$, this paper will study the  asymptotic behaviour of a one-parameter family of solutions to the monopolist's problem. Fix a bounded, open, convex domain $\domain \subseteq \mathbf{R}^n$, $f \in C^{0,1}(\overline{\domain})$ with $f \geq \lambda > 0$, and $\tau \in \mathbf{S}^{n-1}$. Then, for each $a > 0$, let us denote by $u_a$ the minimizer of 
\begin{equation}\label{eq:shiftedfunctional}
    \Phi_a[u] := \int_{\Omega + a \tau}[\frac{1}{2}|Du - x|^2 + u]f(x - a\tau)dx
\end{equation}
over the set
\begin{equation}\label{eq:shiftedadmissibleset}
    \mathcal{U}_a := \{u : \overline{\domain} + a \tau \to \mathbf{R} \ | \ u \geq 0 \text{ is convex}\}. 
\end{equation}
Likewise, let 
\begin{equation}\label{eq:profitofua}
    \Pi_a[u_a] := \int_{\domain + a \tau}|x|^2 f(x - a \tau) dx - \Phi_a[u_a]
\end{equation}
denote the monopolist's profits. Then, $u_a$ is the solution to the monopolist's problem when the distribution of customers $f\, d\mathcal{H}^n|_{\domain}$ is translated by $a \tau$. When $\tau$ lies in the positive orthant $[0, \infty)^n$, shifting by $a\tau$ corresponds to a uniform increase in demand for the monopolist's product; therefore, the limit as $a \to + \infty$ will be referred to as \textit{the seller's market limit}.

We may now state our main results:
\begin{theorem}[Main results]\label{theorem:intromainresults}
    Let $\domain \subseteq \mathbf{R}^n$ be open, bounded, and convex, let $f \in C^{0,1}(\overline{\domain})$ with $f \geq \lambda > 0$, and take $\tau \in \mathbf{S}^{n-1}$. Let $u_a$ minimize \eqref{eq:shiftedfunctional} over \eqref{eq:shiftedadmissibleset}. Denote 
    \[
    d\mu_a := f(x - a\tau)d\mathcal{H}^n|_{\domain + a \tau},
    \]
    where $\mathcal{H}^n$ is the $n$-dimensional volume. Then, there exists a monotonically decreasing function $\Xi(a)$ with $\Xi(a) \to 0$ as $a \to + \infty$ such that the following hold:
    \begin{enumerate}
        \item (Profit formula) \[\Pi_a[u_a] = \frac{a^2}{2}\int_{\domain}f(x)dx + a \Xi(a),\]
        \item (Bound on the exclusion region)
        \[
    \mathcal{H}^n(Du_{a}^{-1}\{0\}) \leq \frac{2}{\lambda - o(a^{-1})}\frac{\Xi(a)}{a}
    \]
    \item ($L^2$ estimate for $Du_a$)
    \[
    ||Du_a - x||_{L^2(\mu_a)} \leq \sqrt{2a \Xi(a)}, \text{ and }
    \]
    \item ($L^2$ estimate for $u_a$) if we assume in addition that $f$ is log-concave, then 
    \[
    ||u_a - \frac{1}{2}(|x|^2 - a^2) + c_a||_{L^2(\mu_a)} \leq \sqrt{2C a \Xi(a)},
    \]
    where $C$ is a Poincar\'{e} constant for $\frac{1}{\int_{\domain}f(x)dx}d\mathcal{H}^n|_{\domain}$, and 
    \[
    c_a := \frac{1}{\int_{\domain}f(x)dx}\left(a\Xi(a) + \frac{1}{2}||Du_a - x||_{L^2(\mu_a)}^2\right).
    \]
    \end{enumerate}
\end{theorem}
These results are proved in Lemma \ref{lemma:profitformula}, Corollary \ref{cor:boundontheexclusionregion}, and Proposition \ref{prop:boundonl2convergence}, respectively. Theorem \ref{theorem:intromainresults} (2) shows that the proportion of consumers that will be priced out of the market under the monopolist's optimal price menu vanishes in the seller's market limit; this provides a counterpoint to Armstrong's desirability of exclusion. Theorem \ref{theorem:intromainresults} (3) and (4) show that $u_a$ can be approximated by $\frac{1}{2}(|x|^2 - a^2) - c_a$ as $a \to + \infty$, providing a first-order asymptotic description of the solution.

Unlike McCann, Rankin, and Zhang's description of the solution in the square example \cite[\S 4]{McCannZhang23+}, \cite[Theorem 1.5]{McCannRankinZhang24+}, which relied on studying Euler-Lagrange equations, the proof of Theorem \ref{theorem:intromainresults} will rely on more indirect methods, including a novel monotonicity formula in the seller's market limit, Theorem \ref{theorem:monotonicityformula}.

\section{The seller's market limit in the monopolist's problem}

In order to study the seller's market limit, it is useful to rescale $u_a$ in order to extract a convergent subsequence as $a \to + \infty$. To this end, we define $v_a : \domain \to \mathbf{R}$ by
\begin{equation}\label{eq:vadefinition}
    v_{a}(x) := \frac{1}{a}u(x  + a \tau),
\end{equation}
then $v_{a}$ will be the minimizer of the functional
\begin{equation}\label{eq:rescaledfunctional}
    \Psi_{a}[v] := \int_{\domain}[\frac{1}{2}|Dv - \frac{x}{a} - \tau|^2 + \frac{1}{a}v]f(x)dx
\end{equation}
over the set
\[
\mathcal{V} := \{v : \overline{\domain} \to \mathbf{R} \ |\ v \geq 0 \text{ is convex}\}. 
\]
We will now study the limits of $v_a$ as $a \to + \infty$. First, we will recall a Lipschitz bound for the solution to the monopolist's problem. Lemma \ref{lemma:lipschitzbound} is attributed to Chon\'{e} in \cite{CarlierLachandRobert01}; we will follow the presentation in \cite[Lemma A.1]{McCannRankinZhang24+}.

\begin{lemma}[Chon\'{e}'s Lipschitz bound]\label{lemma:lipschitzbound}
    Suppose $\domain \subseteq \mathbf{R}^n$ is open, bounded, and convex. Suppose $f \in C^{0,1}(\overline{\domain})$ satisfies $f \geq \lambda > 0$ on $\overline{\domain}$. Let $u: \domain \to \mathbf{R}$ be a nonnegative convex function minimizing \eqref{eq:profitfunctiondefinition} over \eqref{eq:admissiblefunctions}. Then, $u$ is Lipschitz on $\overline{\domain}$ with constant
        \[
        2\sup_{x \in \domain}|x|.
        \]
\end{lemma}

\begin{proof}
    Since $u$ is convex, it is locally Lipschitz in $\domain$, and differentiable almost everywhere. Therefore, it suffices to establish $|Du(x)| \leq 2|x|$ at any point of differentiability $x \in \domain$. Let
    \[
    v(y) := \sup_{x \in \domain}\left(\langle y, x \rangle - u(x)\right)
    \]
    denote the Legendre transform of $u$; see \cite[\S 12]{Rockafellar70}.
    We claim that $v(y) \geq \frac{1}{2}|y|^2$ for every $y$ in $D u(\domain)$. 
    
    If this is not the case, set 
    \[
    w(y) := \max\{v(y), \frac{1}{2}|y|^2\},
    \]
    and let $w^*(x)$ denote its Legendre transform. Then, $w^*$ is convex. Moreover, if $z \in \domain$, then by \cite[Theorem 12.2]{Rockafellar70}
    \begin{align*}
        w^*(z) &= \sup_{y}\left(\langle z, y \rangle - \max\{v(y), \frac{1}{2}|y|^2\}\right) \\
        &= \min\{\sup_y\left(\langle z, y \rangle - \frac{1}{2}|y|^2\right), u(z)\} \\
        &\geq 0.
    \end{align*}
    So, $w^* \geq 0$, and hence $w^*$ is an admissible competitor to $u$. Since we are assuming that $v(y) < \frac{1}{2}|y|^2$ for some $y \in Du(\domain)$, there is some $z \in \domain$ such that $w^*(z) < u(z)$; in particular, this inequality holds on a positive measure subset of $\domain$. Moreover, at points of differentiability at which $w^*(z) \ne u(z)$, we have that $Dw^*(z) = z$. Therefore, we see that
    \[
    \frac{1}{2}|Dw^*(z) - z|^2  + w^*(z) < \frac{1}{2}|Du(z) - z|^2 + u(z)
    \]
    on the positive measure subset of $\domain$ for which $w^*(z) < u(z)$ and $w^*$ and $u$ are differentiable; on the other hand, we have 
    \[
    \frac{1}{2}|Dw^*(z) - z|^2  + w^*(z) \leq \frac{1}{2}|Du(z) - z|^2 + u(z)
    \]
    on all of $\domain$. Therefore, since $f \geq \lambda > 0$, we see by the definition \eqref{eq:profitfunctiondefinition} of $\Phi$ that
    \[
    \Phi(w^*) < \Phi(u),
    \]
    contradicting the minimality of $u$. 

    Now, we have established that $v(y) \geq \frac{1}{2}|y|^2$ for every $y \in Du(\domain)$. By \cite[Theorem 23.5]{Rockafellar70}, we know that 
    \[
    u(x) + v(Du(x)) \leq \langle x, Du(x) \rangle;
    \]
    therefore, using that $v(Du(x)) \geq \frac{1}{2}|Du(x)|^2$ and $u(x) \geq 0$, we conclude that 
    \[
    \frac{1}{2}|Du(x)|^2 \leq \langle x, Du(x) \rangle.
    \]
    Thus, 
    \[
    |Du(x)| \leq 2|x|,
    \]
    as desired. 
\end{proof}

The following simple lemma is well-known: we include its proof for completeness. 
\begin{lemma}\label{lemma:existenceofzeroes}
    Let $v_a$ be the minimizer of \eqref{eq:rescaledfunctional} over $\mathcal{V}$, and assume that $a \geq 1$. Then, there is some $x_a \in \overline{\domain}$ such that $v_a(x_a) = 0$.
\end{lemma}
\begin{proof}
    Suppose that such an $x_a$ does not exist. Then, since $v_a \in C^{0,1}(\domain)$ by Lemma \ref{lemma:lipschitzbound}, and $\overline{\domain}$ is compact, we know that there is some $\eta > 0$ such that $v_a \geq \eta > 0$ on $\domain$. In particular, $v_a - \eta$ is a valid competitor to $v_a$. But we compute
    \begin{align*}
        \Psi_{a}[v_a] - \Psi_{a}[v_a - \eta] &= \frac{\eta}{a}\int_{\domain}f(x)dx \\
        & > 0,
    \end{align*}
    contradicting the minimality of $v_a$.
\end{proof}

\begin{proposition}[Identification of the seller's market limit]\label{prop:identificationofthesellersmarketlimit}
    Take $x_0 \in \partial \domain$ such that $-\tau$ lies in the outward normal cone to $\domain$ at $x_0$. Then, as $a \to + \infty$, $v_{a}$ converges uniformly to 
    \[
    v_{\infty}(x) := \langle \tau , x - x_0\rangle.
    \]
    Moreover, $Dv_{a}$ converges uniformly on compact subsets of $\domain$ to $\tau$. 
\end{proposition}
\begin{proof}  
    By Lemmas \ref{lemma:lipschitzbound} and \ref{lemma:existenceofzeroes}, the family $\{v_{a}\}_{a \geq 1}$ is uniformly bounded and equi-Lipschitz on $\domain$. So, by the Arzel\`{a}-Ascoli theorem, for any sequence $a_n \to + \infty$ we may find a subsequence (not relabelled) and a Lipschitz function $v_{\infty} : \domain \to \mathbf{R}$ such that $v_{a_n} \to v_{\infty}$ uniformly on $\overline{\domain}$. Since convexity is preserved under such limits by \cite[Theorem 10.8]{Rockafellar70}, we know that $v_{\infty}$ is convex and nonnegative. Moreover, the sequence $Dv_{a_n}$ is uniformly bounded on $\overline{\domain}$, and $Dv_{a_n} \to Dv_{\infty}$ pointwise a.e. on $\domain$; the latter claim follows from \cite[Theorem 24.5]{Rockafellar70} since $v_{\infty}$ is differentiable a.e.

    Now, for any convex and nonnegative function $w: \overline{\domain} \to \mathbf{R}$, we have by minimality that
    \[
    \Psi_{a_n}[v_{a_n}] \leq \Psi_{a_n}[w]. 
    \]
    Taking $n \to + \infty$ and applying the dominated convergence theorem, we conclude that
    \[
    \int_{\domain}\frac{1}{2}|Dv_{\infty} - \tau|^2 f(x) dx \leq \int_{\domain}\frac{1}{2}|Dw - \tau|^2 f(x)dx,
    \]
    so we see that $v_{\infty}$ minimizes the functional
    \[
    \Psi_{\infty}[v] := \int_{\domain}\frac{1}{2}|Dv - \tau|^2 f(x)dx
    \]
    over the set of all convex and nonnegative $v: \overline{\domain} \to \mathbf{R}$. In particular, we have that 
    \[
    \int_{\domain}|Dv_{\infty} - \tau|^2 f(x) dx = 0. 
    \]
    Therefore, $Dv_{\infty}(x) = \tau$ for a.e. $x \in \domain$; since $v_{\infty}$ is convex, this implies that $Dv_{\infty}(x) = \tau$ for every $x \in \domain$. 

    Now, by Lemma \ref{lemma:existenceofzeroes} and the fact that $\overline{\domain}$ is compact, we know that there is some $x_0 \in \overline{\domain}$ such that $v_{\infty}(x_0) = 0$. Therefore,
    \[
    v_{\infty}(x) = \langle\tau , x - x_0\rangle . 
    \]
    Moreover, since $v_{\infty} \geq 0$ on $\domain$, we must necessarily have that $-\tau$ lies in the outward normal cone to $\domain$ at $x_0$; notice that $v_{\infty}$ will be independent of the choice of $x_0$ satisfying this property. Therefore, since the initial sequence was arbitrary, we conclude that $v_{a}$ converges uniformly to $v_{\infty}$ as $a \to + \infty$. Since $v_{\infty}$ is differentiable everywhere on $\domain$, \cite[Theorem 25.7]{Rockafellar70} then implies that $Dv_{a}$ converges uniformly on compact subsets of $\domain$ to $\tau$. 
\end{proof}

\section{A monotonicity formula}

We are interested in obtaining control on the rate at which $v_a$ converges to the limit in Proposition \ref{prop:identificationofthesellersmarketlimit}. To this end, we now introduce a monotonicity formula.

\begin{theorem}[Monotonicity formula]\label{theorem:monotonicityformula}
    Let $\domain \subseteq \mathbf{R}^n$ be open, bounded, and convex, and assume that $0 \in \partial \domain$. Let $\tau \in \mathbf{S}^{n-1}$ be such that $-\tau$ lies in the outward normal cone to $\partial \domain$ at $0$. Assume that $f \in C^{0,1}(\domain)$ satisfies $f \geq \lambda > 0$ on $\overline{\domain}$. For each $a > 1$, let $v_{a}$ be the minimizer of \eqref{eq:rescaledfunctional} over all convex and nonnegative $v: \overline{\domain} \to \mathbf{R}$. Define
    \begin{equation}\label{eq:monotonequantitydefinition}
        \Xi(a) := \int_{\domain}\left[ (\langle Dv_a, x \rangle - v_a) - \frac{a}{2}|Dv_a - \tau|^2\right]f(x)dx.
    \end{equation}
    Then, 
    \begin{equation}\label{eq:formulaforderivativeofXi}
        \frac{d}{da}\Xi(a) = -\frac{1}{2}\int_{\domain}|Dv_a - \tau|^2 f(x)dx;
    \end{equation}
    in particular, $\Xi(a)$ is monotonically decreasing. Moreover, we have that 
    \[
    \lim_{a \to + \infty}\Xi(a) = 0. 
    \]
\end{theorem}
\begin{proof}
    First, we claim that 
    \begin{equation}\label{eq:startingderivativeequation}
        \frac{d}{da}\int_{\domain}\frac{1}{2}|Dv_a - \tau|^2f(x) dx - \frac{1}{a}\frac{d}{da}\int_{\domain}\left(\langle Dv_a, x\rangle - v_a\right)f(x) dx = 0. 
    \end{equation}
    Indeed, fix $a > 1$, and let $h \in \mathbf{R}$ be sufficiently small that $a + h \geq 1$. Since 
    $\Psi_{a}[v_a] \leq \Psi_{a}[v_{a + h}]$, we have 
    \[
    \int_{\Omega}\left[\frac{1}{2}(|Dv_a - \tau - \frac{x}{a}|^2 - |Dv_{a + h} - \tau - \frac{x}{a}|^2) + \frac{1}{a}(v_a - v_{a + h})\right]f(x) dx \leq 0,
    \]
    which simplifies to 
    \begin{align*}
        0&\geq \int_{\domain}\frac{1}{2}(|Dv_a - \tau|^2 - |Dv_{a + h} - \tau|^2) f(x) dx \\
        &\qquad - \frac{1}{a}\int_{\domain}(\langle Dv_{a} - Dv_{a+h}, x \rangle - (v_a - v_{a + h}))f(x) dx. 
    \end{align*}
    Instead using that $\Psi_{a + h}[v_a] \geq \Psi_{a + h}[v_{a+h}]$, we get that 
    \begin{align*}
        0&\leq \int_{\domain}\frac{1}{2}(|Dv_a - \tau|^2 - |Dv_{a + h} - \tau|^2) f(x) dx \\
        &\qquad - \frac{1}{a + h}\int_{\domain}(\langle Dv_{a} - Dv_{a+h}, x \rangle - (v_a - v_{a + h})f(x) dx. 
    \end{align*}
    Therefore, if we divide by $h$ and take $h \to 0$, we find that 
    \[
    \frac{d}{da}\int_{\domain}\frac{1}{2}|Dv_a - \tau|^2f(x) dx - \frac{1}{a}\frac{d}{da}\int_{\domain}\left(\langle Dv_a, x\rangle - v_a\right)f(x) dx = 0,
    \]
    as claimed. Now, using \eqref{eq:startingderivativeequation}, we compute that
    \begin{align*}
        \frac{d}{da}\Xi(a) &= -\frac{1}{2}\int_{\domain}|Dv_{a} - \tau|^2 f(x) dx \\
        &\qquad - a\frac{d}{da}\int_{\domain}\frac{1}{2}|Dv_a - \tau|^2f(x) dx + \frac{d}{da}\int_{\domain}\left(\langle Dv_a, x\rangle - v_a\right)f(x) dx \\
        &= -\frac{1}{2}\int_{\domain}|Dv_{a} - \tau|^2 f(x) dx,
    \end{align*}
    proving \eqref{eq:formulaforderivativeofXi}.

    Now, we wish to show that $\lim_{a \to + \infty}\Xi(a) = 0$. Observe that \eqref{eq:formulaforderivativeofXi} implies that 
    \begin{align*}
        \frac{d}{da}\frac{1}{a}\Xi(a) &= -\frac{1}{2a}\int_{\domain}|Dv_a - \tau|^2 f(x)dx - \frac{1}{a^2}\Xi(a) \\
        &= -\frac{1}{a^2}\int_{\domain}(\langle Dv_a , x \rangle - v_a)f(x)dx. 
    \end{align*}
    Therefore, for any $b > a$, we have
    \[
    \frac{1}{b}\Xi(b) - \frac{1}{a}\Xi(a) = -\int_{a}^b \frac{1}{t^2}\int_{\domain}(\langle Dv_t , x \rangle - v_t)f(x)dx dt. 
    \]
    By Proposition \ref{prop:identificationofthesellersmarketlimit}, we know that $Dv_a \to \tau$ uniformly on compact subsets of $\domain$; since $Dv_a$ is uniformly bounded by Lemma \ref{lemma:lipschitzbound}, we may apply the dominated convergence theorem to conclude that 
    \[
    \lim_{b \to + \infty}\frac{1}{b}\Xi(b) = 0. 
    \]
    Therefore, we see that 
    \[
    \frac{1}{a}\Xi(a) = \int_{a}^{\infty}\frac{1}{t^2}\int_{\domain}(\langle Dv_t, x\rangle - v_t)f(x) dx. 
    \]
    In particular, if we let 
    \[
    \beta(a) := \sup_{t \geq a}\left|\int_{\domain}(\langle Dv_t, x \rangle - v_t)f(x)dx\right|, 
    \]
    then we conclude that 
    \begin{equation}\label{eq:boundsonXi}
        - \frac{2\beta(a)}{a} \leq \frac{1}{a}\Xi(a) \leq \frac{2\beta(a)}{a}.
    \end{equation}
    By Proposition \ref{prop:identificationofthesellersmarketlimit}, $v_t \to \langle\tau,  x\rangle $ uniformly on $\domain$, while $Dv_t \to \tau$ locally uniformly on $\domain$. Thus, by the dominated convergence theorem, $\beta(a) \to 0$ as $a \to + \infty$. In particular, by \eqref{eq:boundsonXi}, we conclude that 
    \[
    \lim_{a \to + \infty}\Xi(a) = 0.
    \]
\end{proof}

\section{Bound on $L^2$ convergence and exclusion region}

Using Theorem \ref{theorem:monotonicityformula}, we will now bound the $L^2$ convergence rate of $Dv_a$ to $\tau$ and the volume of the exclusion region $Du_a^{-1}\{0\}$. First, we will prove the profit formula Theorem \ref{theorem:intromainresults} (1); this will provide a reinterpretation of $\Xi(a)$ which will be helpful in bounding the rate of convergence of $Dv_a$ to $\tau$. 

\begin{lemma}[Profit formula]\label{lemma:profitformula}
     Take the assumptions of Theorem \ref{theorem:monotonicityformula}, let $\Pi_a$ be defined as in \eqref{eq:profitofua}, let $\Psi_a$ be defined as in \eqref{eq:rescaledfunctional}, and let $\Xi(a)$ be defined as in \eqref{eq:monotonequantitydefinition}. Then, we have that 
    \begin{equation}\label{eq:reinterpretingXiintermsofPsi}
        \frac{1}{a}\Xi(a) = \Psi_a(\frac{1}{2a}|x|^2 + \langle \tau, x \rangle) - \Psi_a(v_a). 
    \end{equation}
    In particular, this implies that 
   \begin{equation}\label{eq:profitformula}
       \Pi_a[u_a]= \frac{a^2}{2}\int_{\domain}f(x)dx + a \Xi(a). 
   \end{equation}
\end{lemma}
\begin{proof}
    Using the fact that 
    \[
    |Dv_a - \tau|^2 = |Dv_a - \tau - \frac{x}{a}|^2 + 2\langle Dv_a - \tau, \frac{x}{a}\rangle - \frac{|x|^2}{a^2}, 
    \]
    observe that 
    \begin{align*}
        \Xi(a) &= \int_{\domain}\left[(\frac{1}{2a}|x|^2 + \langle \tau, x \rangle - v_a) - \frac{a}{2}|Dv_a - \tau - \frac{x}{a}|^2\right]f(x) dx \\
        &= a\left(\Psi_a(\frac{1}{2a}|x|^2 + \langle \tau, x \rangle) - \Psi_a(v_a)\right),
    \end{align*}
    proving \eqref{eq:reinterpretingXiintermsofPsi}. Now, notice that
    \begin{align*}
        \Psi_a(\frac{1}{2a}|x|^2 + \langle \tau, x \rangle) &=  
        \int_{\domain}(\frac{1}{2a^2}|x|^2 + \langle \tau, \frac{x}{a} \rangle)f(x)dx \\
        &=
        \int_{\domain}(\frac{1}{2}|\frac{x}{a} + \tau|^2 - \frac{1}{2}|\tau|^2)f(x) dx. 
    \end{align*}
    So, multiplying \eqref{eq:reinterpretingXiintermsofPsi} by $a^2$ and applying the change of variables $y = x + a\tau$, we conclude that 
    \[
    a\Xi(a) = \int_{\domain + a \tau}\frac{1}{2}|y|^2 f(y - a \tau) dy - \frac{a^2}{2}\int_{\domain}f(x) dx - \Phi_a[u_a],
    \]
    proving \eqref{eq:profitformula}. 
\end{proof}

Now, let us recall the following lemma. Given a convex function $v \in C^{0,1}(\domain)$, denote by $D^-v \in L^1(\partial \domain, \mathcal{H}^{n-1})$ the \textit{inner boundary trace} of $Dv$; see \cite[Definition 5.3]{EvansGariepy92}, \cite[eq.(4)]{McCannRankinZhang24+}.

\begin{lemma}\label{lemma:firstvariationmeasure}
    Let $\Psi_a$ be as defined in \eqref{eq:rescaledfunctional}, and let $v, w : \domain \to \mathbf{R}$ be arbitrary functions such that $w \in C^{0,1}(\domain)$, $v \in C^{0,1}(\domain) \cap C^{1,1}_{\mathrm{loc}}(\domain)$, and $v$ is convex. Assume that $f \in C^{0,1}(\domain)$ satisfies $f \geq \lambda > 0$ on $\domain$. Define a differential operator $\mathcal{L}_f$ by 
    \begin{equation}\label{eq:Lfdefinition}
        \mathcal{L}_f \varphi := \Delta \varphi + \langle D \varphi, D(\log f)\rangle
    \end{equation}
    for any $\varphi \in C^{\infty}(\domain)$. Then, we have that 
    \begin{align*}
         \Psi_a[v + w] - \Psi_a[v] &= \int_{\domain}w(\mathcal{L}_f( \frac{|x|^2}{2a} + \langle x, \tau \rangle - v) + \frac{1}{a})f(x) dx \\
         &\qquad + \frac{1}{2}\int_{\domain}|Dw|^2 f(x)dx \\
         &\qquad + \int_{\partial \domain}w\langle D^-v - \frac{x}{a} - \tau, \mathbf{n}\rangle  f(x) \, d\mathcal{H}^{n-1} 
    \end{align*}
    where $\mathbf{n}$ is the outward unit normal vector to $\partial \domain$. 
\end{lemma}
\begin{proof}
    Indeed, by definition, we have 
    \begin{align*}
        &\Psi_a[v + w] - \Psi_a[v] \\ 
        &\qquad = \int_{\domain}\left[\frac{1}{2}(|Dv + Dw - \frac{x}{a} - \tau|^2 - |Dv - \frac{x}{a} - \tau|^2) + \frac{w}{a}\right]f(x) dx \\
        &\qquad = \int_{\domain}\left[\langle Dw,  Dv - \frac{x}{a} - \tau \rangle + \frac{1}{2}|Dw|^2 + \frac{w}{a}\right]f(x) dx.
    \end{align*}
    Observe that 
    \begin{align*}
        \mathrm{div}(wf(Dv - \frac{x}{a} - \tau)) &= \langle Dw, Dv - \frac{x}{a} - \tau \rangle f - wf \mathcal{L}_f( \frac{1}{2a}|x|^2 + \langle x, \tau \rangle - v).
    \end{align*}
    So, by the divergence theorem \cite[Theorem 5.6]{EvansGariepy92}, we have 
    \begin{align*} \int_{\domain}\langle Dw,  Dv - \frac{x}{a} - \tau \rangle f(x) dx &= \int_{\domain}w (\mathcal{L}_f( \frac{|x|^2}{2a} + \langle x, \tau \rangle  - v))f(x) dx \\
    &\qquad + \int_{\partial \domain}w\langle D^-v - \frac{x}{a} - \tau, \mathbf{n}\rangle  f(x) \, d\mathcal{H}^{n-1};
    \end{align*}
    this proves the desired formula. 
\end{proof}

We may now quantify the $L^2$ convergence rate of $v_a$ to $v_{\infty}$. Recall that a function $f$ is called log-concave if $f(x) = e^{- g(x)}$ for some convex function $g$. 

\begin{proposition}[Bound on $L^2$ convergence]\label{prop:boundonl2convergence}
    Take the assumptions of Theorem \ref{theorem:monotonicityformula}, and let $\Xi(a)$ be defined as in \eqref{eq:monotonequantitydefinition}. Then, for all $a \geq 1$, we have that 
    \begin{equation}\label{eq:boundonl2convergence}
       \frac{1}{2}\int_{\domain}|Dv_a - \tau - \frac{x}{a}|^2 f(x) dx \leq \frac{\Xi(a)}{a}.
    \end{equation}
    If we assume in addition that $f$ is log-concave, then we have that 
    \begin{equation}\label{eq:L2convergenceofva}
        \int_{\domain}|v_a - \frac{1}{2a}|x|^2 - \langle \tau, x \rangle + \frac{1}{a}c_a|^2 f(x) dx \leq 2C\frac{\Xi(a)}{a},
    \end{equation}
    where $C$ is a Poincar\'{e} constant for $\frac{1}{\int_{\domain}f(x)dx}fd \mathcal{H}^n|_{\domain}$, and
    \[
    c_a := \frac{1}{\int_{\domain}f(x) dx}\left(a\Xi(a) + \frac{a^2}{2}\int_{\domain}|Dv_a - \tau - \frac{x}{a}|^2 f(x) dx\right).
    \]
\end{proposition}
\begin{proof}
    Apply Lemma \ref{lemma:firstvariationmeasure} with $v = v_a$ and $w = \frac{1}{2a}|x|^2 + \langle \tau, x \rangle - v_a$. Then, by \eqref{eq:reinterpretingXiintermsofPsi}, we have
    \begin{align*}
        \Xi(a) &= a\int_{\domain}w(\mathcal{L}_f w + \frac{1}{a})f(x)dx \\
        &\qquad + \frac{a}{2}\int_{\domain}|Dw|^2 f(x) dx + a\int_{\partial \domain}w\langle D^-v_a - \frac{x}{a} - \tau,  \mathbf{n}\rangle f(x)\, d\mathcal{H}^{n-1}.
    \end{align*}
   Observe that for each $t \in [0,1]$, the function
    \[
    v_a + tw = \frac{t}{2a}|x|^2 + t\langle \tau, x \rangle +(1 - t)v_a
    \]
    is convex and nonnegative on $\domain$. Applying Lemma \ref{lemma:firstvariationmeasure}, we have that 
    \begin{align*}
        &\lim_{t \to 0^+}\frac{1}{t}(\Psi_a[v_a + tw] - \Psi_a[v_a]) \\
        &\qquad = \int_{\domain}w(\mathcal{L}_fw + \frac{1}{a})f(x)dx + \int_{\partial \domain}w\langle D^-v_a - \frac{x}{a} - \tau,  \mathbf{n}\rangle f(x)\, d\mathcal{H}^{n-1},
    \end{align*}
    so by the minimality of $v_a$, we conclude that
    \[
    \int_{\domain}w(\mathcal{L}_fw + \frac{1}{a})f(x)dx + \int_{\partial \domain}w\langle D^-v_a - \frac{x}{a} - \tau, \mathbf{n}\rangle f(x)\, d\mathcal{H}^{n-1} \geq 0. 
    \]
    Therefore, we see that
    \[
    \Xi(a) \geq \frac{a}{2}\int_{\domain}|Dw|^2 f(x) dx,
    \]
    which yields \eqref{eq:boundonl2convergence} since $w = \frac{1}{2a}|x|^2 + \langle \tau, x \rangle - v_a$. 

Now, assume in addition that $f$ is log-concave. Then, the Poincar\'{e} inequality for log-concave probability measures \cite{Bobkov99} implies that there exists some constant $C = C(f, \domain)$ such that
    \[
    \mathrm{Var}_{f}(\frac{1}{2a}|x|^2 + \langle \tau, x \rangle - v_a) \leq   \frac{C}{\int_{\domain}f(x)dx}\int_{\domain}|Dv_a - \tau - \frac{x}{a}|^2 f(x) dx,
    \]
    where
    \[
    \mathrm{Var}_f(g) := \frac{1}{\int_{\domain}f(x)dx}\int_{\domain}\left|g - \frac{1}{\int_{\domain}f(x)dx}\int_{\domain}g(x) f(x) dx\right|^2 f(x)dx.
    \]
    Combining this with \eqref{eq:boundonl2convergence}, we have that 
    \begin{equation}\label{eq:variancebound}
        \mathrm{Var}_f(\frac{1}{2a}|x|^2 + \langle \tau, x \rangle - v_a) \leq \frac{2C}{\int_{\domain}f(x)dx} \frac{\Xi(a)}{a}.
    \end{equation}
    Now, using \eqref{eq:reinterpretingXiintermsofPsi}, observe that 
    \begin{align*}
        \int_{\domain}(\frac{1}{2a}|x|^2 + \langle \tau, x \rangle - v_a) f(x) dx &= \Xi(a) + \frac{a}{2}\int_{\domain}|Dv_a - \frac{x}{a} - \tau|^2 f(x) dx \\
        &= \frac{c_a}{a}\int_{\domain}f(x)dx. 
    \end{align*}
    Combining this with \eqref{eq:variancebound}, we conclude that
    \[
    \int_{\domain}|\frac{1}{2a}|x|^2 + \langle \tau, x \rangle - \frac{c_a}{a} - v_a|^2 f(x) dx \leq 2C \frac{\Xi(a)}{a},
    \]
    as desired. 
\end{proof}

Recall that the set $Du^{-1}\{0\}$ of consumers who are priced out of the market is referred to as the \textit{exclusion region}. We now show that in the seller's market limit, the volume of the exclusion region vanishes; moreover, we will quantify the rate at which it vanishes.
\begin{corollary}[Bound on the exclusion region]\label{cor:boundontheexclusionregion}
    Take the assumptions of Theorem \ref{theorem:monotonicityformula}. Then, we have 
    \[
    \mathcal{H}^n(Du_{a}^{-1}\{0\}) \leq \frac{2}{\lambda - o(a^{-1})}\frac{\Xi(a)}{a}
    \]
    as $a \to +\infty$.
\end{corollary}
\begin{proof}
    Recall that $f \geq \lambda >0 $ on $\overline{\domain}$. Since $|\tau| = 1$, we have that
    \begin{align*}
        \frac{\lambda}{2} \mathcal{H}^n(Du_{a}^{-1}\{0\}) &\leq \frac{1}{2}\int_{Du_{a}^{-1}\{0\} - a \tau}|Dv_a - \tau|^2 f(x)dx \\
        &= \frac{1}{2}\int_{Du_{a}^{-1}\{0\} - a \tau}|Dv_a - \frac{x}{a} - \tau|^2 f(x)dx  + o(a^{-1})\mathcal{H}^n(Du_a^{-1}\{0\}),
        \end{align*}
    so by \eqref{eq:boundonl2convergence} we conclude that
    \[
    (\frac{\lambda}{2} - o(a^{-1}))\mathcal{H}^n(Du_{a}^{-1}\{0\}) \leq \frac{\Xi(a)}{a},
    \]
    as was claimed. 
\end{proof}

\begin{remark}\label{remark:comparisonofexclusionregionboundwiththesquare}
    In the prototypical case of the square $[a,a + 1]^2 \subseteq \mathbf{R}^2$ with constant density $f \equiv 1$, \cite[Lemma 4.7]{McCannZhang23+} and \cite[Theorems 1.1, 1.5]{McCannRankinZhang24+} show that when $a \geq \frac{7}{2} - \sqrt{2}$, the exclusion region $Du_{\sqrt{2}a}^{-1}\{0\} \subseteq [a, a + 1]^2$ is given by an isosceles right triangle occupying the bottom-left corner of $[a, a+1]^2$ whose side length along the edges of the square is
    \[
    L = \frac{- 2a + \sqrt{4a^2 + 6}}{3};
    \]
    see \cite[Figure 1(c)]{McCannRankinZhang24+}.
    Therefore, we compute 
    \begin{align*}
        \mathcal{H}^2(Du_{\sqrt{2}a}^{-1}\{0\}) &= \frac{1}{2}\left(\frac{-2a + \sqrt{4a^2 + 6}}{3}\right)^2 \\
        &= \frac{4a^2}{9}\left(\sqrt{1 + \frac{3}{2a^2}} - 1\right)^2 \\
        &= \frac{1}{4}a^{-2} + O(a^{-4}).
    \end{align*}
    By Corollary \ref{cor:boundontheexclusionregion}, this implies that in the case of the square with constant density,
    \[
    \Xi(a) \geq \frac{1}{4}a^{-1} + o(a^{-2}).
    \]
\end{remark}

It is an interesting problem to quantify the rate at which $\Xi(a)$ vanishes in the seller's market limit. Remark \ref{remark:comparisonofexclusionregionboundwiththesquare} provides the lower bound $\Xi(a) \geq \frac{1}{4}a^{-1} + o(a^{-2})$ in the case of the square with constant density; it is natural to ask whether one can provide a corresponding upper bound $\Xi(a) \leq Ca^{-1} + o(a^{-1})$.

\end{document}